\documentclass[11pt]{article}

\usepackage[a4paper]{anysize}\marginsize{3.5cm}{3.5cm}{1.3cm}{2.5cm}
\pdfpagewidth=\paperwidth \pdfpageheight=\paperheight

\usepackage[latin1]{inputenc}
\usepackage{amssymb}

\makeatletter\renewcommand\@seccntformat[1]{\csname the#1\endcsname.\enspace}\makeatother
\renewenvironment{abstract}{\begin{quote}\hrulefill\par\footnotesize\textbf{\abstractname.}}{\par\vskip-0.5\baselineskip\hrulefill\end{quote}}

\newtheorem{introtheorem}{Theorem}

\newtheorem{thm}{Theorem}[section]

\newtheorem{lemma}[thm]{Lemma}
\newtheorem{proposition}[thm]{Proposition}

\newcommand\mkthm[2]{\newenvironment{#1}{\begin{#2}\rm}{\end{#2}}}
 \mkthm{definition}{tdefinition}
         \mkthm{remark}{tremark}
       \mkthm{example}{texample}
     \mkthm{question}{tquestion}

\newtheorem{thevarthm}[thm]{\varthmname}

\newenvironment{varthm*}[1]{\trivlist\item[]{\bf #1.}\it}{\endtrivlist}

\newenvironment{proof}[1][Proof]{\trivlist\item[\hskip\labelsep{\textit{#1.}}]}{\hspace*{\fill}$\Box$\endtrivlist}

\renewcommand\O{\mathcal O}
\renewcommand\bar{\overline}

\renewcommand\ge{\geqslant}
\renewcommand\le{\leqslant}

\newcommand\keywords[1]{{\renewcommand\thefootnote{}\footnotetext{\textit{Keywords:} #1.}}}
\newcommand\subclass[1]{{\renewcommand\thefootnote{}\footnotetext{\textit{Mathematics Subject Classification (2010):} #1.}}}

\newcommand\Q{\mathbb Q}
\newcommand\Z{\mathbb Z}

\newcommand\be[1][@{\;}r@{\;}c@{\;}l@{\;}l@{\;}]{$$\everymath{\displaystyle}\renewcommand\arraystretch{1.2}\begin{array}{#1}}
\newcommand\ee{\end{array}$$}

\newcommand\tfrac[2]{{\textstyle\frac{#1}{#2}}}
\newcommand\compact{\itemsep=0cm \parskip=0cm}
\newcommand\set[1]{\left\{#1\right\}}
\newcommand\with{\,\,\vrule\,\,}

\newcommand\tmod[1]{\;({\rm mod}~#1)}
\newenvironment{bycases}{\left\{\begin{array}{@{}l@{\quad}l}}{\end{array}\right.}
\newcommand\isom{\simeq}  
\newcommand\tensor{\otimes}
\newcommand\eqnref[1]{(\ref{#1})}
\newcommand\newop[2]{\newcommand#1{\mathop{\rm #2}\nolimits}}
\newcommand\renewop[2]{\renewcommand#1{\mathop{\rm #2}\nolimits}}

\newop\Diag{Diag}
\newop\Hom{Hom}
\newop\GL{GL}
\newop\LCM{LCM}
\newop\Vol{Vol}
\newop\Pos{Pos}
\newop\Amp{Amp}
\newop\Bigcone{Big}
\newop\Nef{Nef}
\newop\NS{NS}
\newop\Fix{Fix}
\renewop\Re{Re}
\newop\mult{mult}
\newop\Pic{Pic}

\newop\End{End}

\newcommand\eps{\varepsilon}
\newcommand\Pell[1]{\mathop{\rm Pell}(#1)}

\begin{document}

   \title{On the integrality of Seshadri constants \\ of abelian surfaces}
   \author{\normalsize Thomas Bauer, Felix Fritz Grimm, Maximilian Schmidt}
   \date{\normalsize August 26, 2019}
   \maketitle
   \thispagestyle{empty}
   \subclass{%
      14C20, 
      14H52, 
      14Jxx, 
      14Kxx  
   }

\begin{abstract}
   In this paper we consider the question of when Seshadri
   constants on abelian surfaces are integers.
   Our first result concerns
   self-products
   $E\times E$ of elliptic curves:
   If $E$ has complex multiplication in
   $\Z[i]$ or in $\Z[\frac12(1+i\sqrt3)]$
   or if $E$ has no complex multiplication at all, then
   it is known that for every ample line bundle
   $L$ on $E\times E$,
   the Seshadri constant
   $\eps(L)$ is an integer.
   We show that, contrary to what one might expect,
   these are in fact the only elliptic curves
   for which this integrality statement holds.
   Our second result answers the question how
   -- on any abelian surface~--
   integrality of Seshadri constants is related to elliptic curves.
\end{abstract}


\section{Introduction}

   For an ample line bundle $L$
   on a smooth projective variety $X$, the \emph{Seshadri constant}
   of $L$ at a point $x\in X$ is by definition
   the real number
   \be
      \eps(L,x)=\inf\set{\frac{L\cdot C}{\mult_x(C)}\with C\mbox{ irreducible curve through } x}
      \,.
   \ee
   On abelian varieties, where
   they are independent of
   the chosen point $x$,
   these invariants
   have been the focus of a great deal of attention
   \cite{Lazarsfeld:periods,Nakamaye:Seshadri-abelian-var,Bauer:periods}.
   In the two-dimensional case, they
   are completely understood in the case when the Picard number
   of the abelian surface
   is one \cite{Bauer:sesh-alg-sf}. At the other extreme,
   self-products $E\times E$ of elliptic curves were studied in
   \cite{Bauer-Schulz}, where $E$ is either an elliptic curve
   without complex multiplication or
   with $\End(E)=\Z[i]$ or $\End(E)=\Z[\frac12(1+i\sqrt3)]$.
   In those cases, the Seshadri constants $\eps(L)$ of all ample
   line bundles $L$ on $E\times E$ were found to be integers~--
   they are
   in fact computed by elliptic curves.
   It is natural to expect that this should in effect
   hold on all surfaces $E\times E$, where $E$ is an elliptic
   curve. Surprisingly, however,
   it turns out that the exact opposite is the case:
   Fractional Seshadri constants do occur on
   all
   self-products $E\times E$
   except for the ones considered so far.
   The following theorem
   provides the complete picture:

\begin{introtheorem}\label{thm:intro-integer}
   Let $E$ be an elliptic curve with complex multiplication.
   Then the following conditions are equivalent:
   \begin{itemize}\compact
   \item[\rm(i)]
      For every ample line bundle $L$ on $E\times E$, the Seshadri
      constant $\eps(L)$ is an integer.
   \item[\rm(ii)]
      Either $\End(E)=\Z[i]$ or $\End(E)=\Z[\frac12(1+i\sqrt3)]$.
   \end{itemize}
\end{introtheorem}

   For the proof of Theorem~\ref{thm:intro-integer} we will first
   establish how integrality is related to elliptic curves.
   One direction is obvious:
   If all Seshadri constants on a given abelian surface are computed by
   elliptic curves, then certainly those numbers are all
   integers. It is however less clear to what extent the converse statement holds
   true. The following theorem answers this question;
   it holds on any abelian surface, regardless of whether
   it splits as a product or not.

\begin{introtheorem}\label{thm:intro-computed-elliptic}
   Let $X$ be an abelian surface. The following conditions are
   equivalent:
   \begin{itemize}\compact
   \item[\rm(i)]
      For every ample
      line bundle $L$ on $X$, the Seshadri constant $\eps(L)$ is an
      integer.
   \item[\rm(ii)]
      For every ample
      line bundle $L$ on $X$, either $\eps(L)=\sqrt{L^2}$
      and $\sqrt{L^2}$ is an integer,
      or $\eps(L)$ is computed by an elliptic curve, i.e., there exists an
      elliptic curve $E\subset X$ such that
      \be
         \eps(L) = L\cdot E
         \,.
      \ee
   \end{itemize}
\end{introtheorem}

   If one is interested in constructing
   explicit examples of line bundles with fractional Seshadri constants
   on products $E\times E$,
   then a natural approach is to look for irreducible principal
   polarizations on these surfaces.
   In other words, one asks under which circumstances
   $E\times E$ is the Jacobian
   of a smooth genus 2 curve.
   This question
   was first studied by
   Hayashida and Nishi \cite{Hayashida-Nishi:genus-two}
   in the case where the Endomorphism ring is the maximal order
   in the Endomorphism algebra.
   We extend their result in Prop.~\ref{prop:exists-principal-polarization}
   to cases which include non-maximal orders.

   Looking at Theorem~\ref{thm:intro-computed-elliptic}, one might be
   tempted to hope that an analogous equivalence might hold
   for \emph{each individual} line bundle. However, as we will
   show in Prop.~\ref{prop:primitive-lb},
   $\eps(L)$ can be an integer without this being accounted for
   by the conditions in (ii).

   We work throughout over the field of complex numbers.


\section{Background}

   We will use
   a number of previous results concerning Seshadri constants
   on abelian surfaces
   in the
   proof of Thm.~1 and Thm.~2.
   In this preliminary section
   we briefly review some of
   these results. For more general background on Seshadri constants
   also see \cite[Chapt.~5]{Lazarsfeld:pos} and \cite{Bauer-et-al:pimer}

\paragraph{\it Submaximal divisors.}
   For any ample line bundle $L$ an a smooth projective surface $X$
   and any point $x\in X$,
   one has the basic bound
   $\eps(L,x)\le\sqrt{L^2}$.
   A divisor $D$ on $X$ is called \emph{$L$-submaximal}
   at $x$ if its
   Seshadri quotient ${L\cdot D}/{\mult_0 D}$ is
   strictly smaller than $\sqrt{L^2}$.
   In other words, a divisor $D$ is $L$-submaximal if it forces $\eps(L,x)$ to be
   smaller than the theoretical upper bound $\sqrt{L^2}$.
   It is a crucial observation that
   if $D$ is an $L$-submaximal divisor that belongs to the
   linear series $|kL|$ for some integer $k\ge 1$, then
   every irreducible curve $C$ that is $L$-submaximal
   at $x$ must occur as a component of $D$.
   (see \cite[Lemma 5.2 and Lemma~6.2]{Bauer:sesh-alg-sf}).
   It is for this fact that
   submaximal divisors are in many cases instrumental
   to explicitly determining Seshadri constants.
   One such situation is as follows:
   Suppose that $C$ is an irreducible curve that is $L$-submaximal at
   $x$ for some ample line bundle $L$.
   If the line bundle $\O_X(C)$ is ample, then by \cite[Prop.~1.2]{Bauer-Schulz}
   the curve $C$ is also
   $\O_X(C)$-submaximal at $x$,
   and in fact $C$ \emph{computes} $\eps(\O_X(C),x)$ in the sense
   that
   \be
      \eps(\O_X(C),x) = \frac{\O_X(C)\cdot C}{\mult_x C}
      \,.
   \ee

\paragraph{\it Symmetric divisors on abelian surfaces.}
   Consider now an abelian surface $X$.
   A divisor $D$ on $X$ is \emph{symmetric}, if $D$
   is invariant under the involution $x\mapsto -x$.
   Similarly, a line bundle $L$ is \emph{symmetric},
   if $(-1)^*L=L$ in $\Pic(X)$.
   Symmetric line bundles enjoy important properties:
   The sixteen halfperiods on $X$ are
   divided into \emph{even} and \emph{odd} halfperiods
   (with respect to $L$), and if $D$ is a symmetric divisor
   with $\O_X(D)=L$, then the multiplicities
   of $D$ at either of these sets of halfperiods
   are all even or all odd.
   (See \cite[Sect.~4.7]{BL:CAV} for these facts and for
   further properties of symmetric line bundles.)

   Let $L$ be an ample line bundle on an abelian surface.
   As far as Seshadri constants are concerned we may assume
   that $L$ is symmetric, since $\eps(L)$
   depends only on the algebraic equivalence class
   and every algebraic equivalence class contains contains
   symmetric line bundles.
   It was proven in \cite{Bauer-Szemberg:appendix} that
   if $L$ is primitive and $\sqrt{L^2}$ is irrational, then
   $\eps(L)<\sqrt{L^2}$.
   This is shown by constructing a symmetric submaximal divisor,
   the
   \emph{Pell divisor}, $\Pell L$ in $|2kL|$ with multiplicity at least $2\ell$,
   where $(k,\ell)$ is the minimal solution of the Pell equation
   $\ell^2-L^2k^2=1$.
   One has
   \be
      \frac{L\cdot\Pell L}{\mult_0{\Pell L}}\le L^2\cdot\frac k\ell <\sqrt{L^2}
      \,.
   \ee


\section{Integral Seshadri constants}

   In this section we prove Theorem~\ref{thm:intro-computed-elliptic}.
   As mentioned in the previous section,
   one has $\eps(L)\le\sqrt{L^2}$ for any ample line
   bundle.
   We start by giving an example showing that
   in condition (ii) of the theorem it can in fact happen that
   $\eps(L)=\sqrt{L^2}$, even though $\eps(L)$ is not computed by an
   elliptic curve.

\begin{example}
   For any positive integer $n$ consider
   a polarized abelian surface $(X,L)$ of type
   $(1,2n^2)$ with $\rho(X)=1$.
   As $L^2=4n^2$ is a perfect square, one has
   $\eps(L)=\sqrt{L^2}=2n$
   by \cite[Prop.~1]{Steffens:remarks},
   but of course $\eps(L)$ is not
   computed by an elliptic curve, since there are no such curves
   on $X$.
\end{example}

\begin{proof}[Proof of Theorem~\ref{thm:intro-computed-elliptic}]
   The implication (ii) $\Rightarrow$ (i) being obvious, let us
   suppose (i). Assume by way of contradiction that there are
   ample line bundles $L$ on $X$ whose Seshadri constant is
   less than $\sqrt{L^2}$ and
   not
   computed by elliptic curves.
   Consider a primitive such line bundle $L$.
   Replacing $L$ by an algebraically equivalent line bundle,
   we may assume that $L$ is symmetric.
   We will now make use of the Pell divisor of $L$, i.e.,
   the divisor $D=\Pell L\in|2kL|$ with $\mult_0(D)\ge 2\ell$,
   where $(k,\ell)$ is the minimal solution of Pell's equation
   \be
      \ell^2 - 2d k^2 = 1
      \,,
   \ee
   having the property that
   \be
      \frac{L\cdot D}{\mult_0(D)} < \sqrt{L^2}
   \ee
   (see \cite{Bauer-Szemberg:appendix}).
   It follows from \cite[Lemma~5.2]{Bauer:sesh-alg-sf}
   that every irreducible curve computing $\eps(L)$ is a component of $D$.
   Let $C$ be one of these curves (which by
   assumption is not elliptic).
   As $C$ is sub-maximal for $L$,
   it follows from \cite[Prop.~1.2]{Bauer-Schulz} that
   $C$ computes its own Seshadri constant $\eps(\O_X(C))$.
   The curves $C$ and $(-1)^* C$ have the same multiplicity at
   the point $0$
   and they are algebraically equivalent.
   Therefore, by applying \cite[Lemma~5.2]{Bauer:sesh-alg-sf}
   to the bundle $\O_X(C)$, we see that
   these two curves must coincide, i.e., that $C$ is symmetric.
   So $C$ descends to a curve $\bar C$ on the smooth Kummer surface of $X$.
   With an argument as in the proof of
   \cite[Thm.~6.1]{Bauer:sesh-alg-sf}, this curve $\bar C$
   must be a $(-2)$-curve. (Otherwise $C$ would move in a
   pencil of $L$-submaximal curves, but there can only be
   finitely many of those.)
   The upshot of these considerations is that
   the multiplicities $m_i=\mult_{e_i}(C)$ of $C$ at the sixteen
   halfperiods $e_i$ of $X$ satisfy the equation
   \begin{equation}\label{eqn:minus-four}
      C^2 - \sum_{i=1}^{16} m_i^2 = -4
      \,.
   \end{equation}
   On the other hand, one has
   \begin{equation}\label{eqn:negative}
      C^2 - m_1^2 < 0
   \end{equation}
   (where $m_1=\mult_0(C)$), since otherwise
   \be
      \frac{L\cdot C}{m_1}
      \ge \frac{\sqrt{L^2}\sqrt{C^2}}{m_1}
      \ge \frac{\sqrt{L^2}\cdot m_1}{m_1}
      = \sqrt{L^2}
   \ee
   contradicting the fact that $C$ is submaximal for $L$.
   We claim now that
   \begin{equation}\label{eqn:differences}
      C^2 - m_1^2 = -1 \quad\mbox{or}\quad
      C^2 - m_1^2 = -4
      \,.
   \end{equation}
   For the proof of \eqnref{eqn:differences},
   note first that,
   by \eqnref{eqn:minus-four} and
   \eqnref{eqn:negative}, the only other possibilities for
   $C^2-m_1^2$ are $-2$ and $-3$.
   In the case where $C^2-m_1^2=-2$,
   we see that the number $m_1$ must be even and
   we have $\sum_{i=2}^{16}m_i^2=2$ by \eqnref{eqn:minus-four}. So there are
   exactly two half-periods at which $C$ has odd multiplicity.
   But this cannot happen since a symmetric divisor can only have $4,6,10$ or $12$
   half-periods with odd multiplicity (see  \cite[Sect.~2, Cor.~3]{Mumford:EDAV1} or \cite[Prop.~4.7.5]{BL:CAV}).
   In the other case, $C^2-m_1^2=-3$, the number $m_1$ is odd and
   we have $\sum_{i=2}^{16}m_i^2=1$ by \eqnref{eqn:minus-four}.
   This leads to the same kind of contradiction as before.

   We know that $C$ computes its own Seshadri constant, i.e.,
   \be
      \eps(\O_X(C))=\frac{C\cdot C}{\mult_0 C}
      \,.
   \ee
   But by \eqnref{eqn:differences}, this number
   equals
   \be
      \frac{m_1^2 - s}{m_1} = m_1 - \frac s{m_1}
   \ee
   where $s=1$ or $s=4$.
   As by assumption $\eps(\O_X(C))$ is a positive integer,
   this means that necessarily
   \be
       m_1=4 \quad\mbox{and}\quad C^2=12
       \,.
   \ee
   In this case
   the multiplicities $m_i$ at the non-zero half-periods are all
   zero.
   So in particular, all
   multiplicities $m_i$ are even.
   Therefore the line bundle $\O_X(C)$ is totally
   symmetric, and hence it is a square of another line bundle
   (see \cite[Sect.~2, Cor.~4]{Mumford:EDAV1}).
   But because of $C^2=12$ this is impossible. So we have
   arrived at a
   contradiction, and this completes the proof of the theorem.
\end{proof}


\section{Products of elliptic curves with complex multiplication}

   In this section we prove Theorem~\ref{thm:intro-integer}.
   Let $E$ be an elliptic curve that has complex multiplication,
   i.e., $\End(E)\tensor\Q=\Q(\sqrt{d})$ for some square-free
   integer $d<0$.
   The endomorphism ring $\End(E)$ is then an order in $\Q(\sqrt{d})$,
   and hence it is of the form
   \be
      \End(E)\isom\Z[f\omega]
   \ee
   where $f$ is a positive integer and
   \be
      \omega=
      \begin{bycases}
         \sqrt d            & \mbox{ if } d\equiv 2,3 \tmod 4 \\
         \tfrac12(1+\sqrt d) & \mbox{ if } d\equiv 1 \tmod 4 \,.
      \end{bycases}
   \ee
   It turns out that for our purposes it is more practical
   to use an equivalent but slightly different description:
   We write
   $\End(E)=\Z[\sigma]$, where
   \be
      \sigma                & = & \sqrt{-e}             & \mbox{ with } e \in \mathbb N \\
      \mbox{or}\quad \sigma & = & \tfrac12(1+\sqrt{-e}) & \mbox{ with } e \in \mathbb N \mbox{ and } e\equiv 3 \tmod 4
      \,.
   \ee
   On the product surface $E\times E$, denote by $F_1=\{0\}\times E$ and $F_2=E\times \{0\}$ the fibers of the
   projections, by $\Delta$ the diagonal, and by
   $\Gamma$ the graph of the endomorphism corresponding to
   $\sigma$.
   The classes of these four curves form a basis of
   $\NS(E\times E)$
   (see
   \cite[Thm.~22]{Weil:var-ab}
   or \cite[Thm.~11.5.1]{BL:CAV}).

\begin{proposition}\label{prop:intersection-matrix}
   Let $E$ be an elliptic curve with complex multiplication.
   Write
   $\End(E)=\Z[\sigma]$ with $\sigma$ as above.
   Then the intersection matrix of $(F_1,F_2,\Delta,\Gamma)$ is
   \be
      \left(
      \begin{array}{cccc}
         0 & 1 & 1 & 1 \\
         1 & 0 & 1 & |\sigma|^2 \\
         1 & 1 & 0 & |1-\sigma|^2 \\
         1 & |\sigma|^2 & |1-\sigma|^2 & 0
      \end{array}
      \right)
      \,.
   \ee
\end{proposition}

\begin{proof}
   All four curves are elliptic, so we have
   \be
      F_1^2=F_2^2=\Delta^2=\Gamma^2=0
      \,.
   \ee
   As each curve intersects the other ones transversely,
   it is enough to count the number of intersection points.
   So we have
   \be
      F_1\cdot F_2=F_1\cdot \Delta = F_1\cdot\Gamma=F_2\cdot \Delta=1
      \,,
   \ee
   since these curves intersect only in the origin.
   For $F_2$ and $\Gamma$ one has
   \be
      F_2\cdot\Gamma=\#\{(x,0)\,|\,x\in E\}\cap\{(x,\sigma x)\,|\,x\in E\}
      \,,
   \ee
   and this shows that we need to count the
   number of solutions $x\in E$ of the equation $\sigma x = 0$.
   But this number equals the degree of the isogeny $\sigma: E\to E$,
   and so we get
   \be
      F_2\cdot\Gamma = \deg \sigma = |\sigma|^2
      \,.
   \ee
   Finally, for $\Delta$ and $\Gamma$ we have
   \be
      \Delta\cdot\Gamma=\#\{(x,x)\,|\,x\in E\}\cap\{(x,\sigma x)\,|\,x\in E\}
      \,,
   \ee
   and this is the number of fixed points of the isogeny $\sigma$.
   Hence by the Holomorphic Lefschetz Fixed-Point Formula \cite[Thm.~13.1.2]{BL:CAV}, we have
   \be
      \Delta\cdot\Gamma=\# \Fix(\sigma) = |1-\sigma|^2
      \,,
   \ee
   and this concludes the proof of the proposition.
\end{proof}

   We will need an explicit description of the ample cone of $E\times E$:

\begin{proposition}\label{prop:ample-cone}
   Let $E$ be an elliptic curve with complex multiplication.
   Write $\End(E)=\Z[\sigma]$ as above.
   A line bundle
   \be
      L=\mathcal O_{E\times E}(a_1F_1+a_2F_2+a_3\Delta +a_4\Gamma)
   \ee
   is ample if and only if the two inequalities
   \be
      a_1+a_2+2a_3+(|\sigma|^2+1)a_4                                     & > & 0 \\
      a_1a_2+a_1a_3+a_1a_4+a_2a_3+|\sigma|^2a_2a_4+|1-\sigma|^2a_3a_4 & > & 0
   \ee
   are satisfied.
\end{proposition}

\begin{proof}
   This follows from the fact that a line bundle $L$ is ample if and only
   if both $L^2$ and the intersection of $L$ with the ample line bundle
   $\mathcal{O}_{E\times E}(F_1+F_2)$ are positive
   (by the improvement of the Nakai-Moishezon criterion
   valid on abelian varieties
   \cite[Cor.~4.3.3]{BL:CAV}).
\end{proof}

   Next, we apply a change of basis to the Néron--Severi group
   to make calculations easier by choosing two basis elements
   which are orthogonal to $F_1$ and $F_2$.
   We define $\nabla:= \Delta - F_1 - F_2$ and $\Sigma:= \Gamma - |\sigma|^2F_1 - F_2$.
   The intersection matrix of $(F_1,F_2,\nabla,\Sigma)$ is then
   \be
      \left(
         \begin{array}{cccc}
            0 & 1 & 0 & 0 \\
            1 & 0 & 0 & 0 \\
            0 & 0 & -2 & -2 \Re(\sigma) \\
            0 & 0 & -2 \Re(\sigma) & -2|\sigma|^2
         \end{array}
      \right)
      \,.
   \ee
   In terms of this new basis, the ampleness condition
   for a line bundle
   \be
      L=\mathcal O_{E\times E}(a_1F_1+a_2F_2+a_3\nabla+a_4\Sigma)
   \ee
   is expressed by the two inequalities
   \be
      a_1+a_2                                        & > & 0 \\
      a_1a_2-a_3^2-|\sigma|^2a_4^2-2\Re(\sigma)a_3a_4 & > & 0 \,.
   \ee

   It was shown in \cite{Bauer-Schulz} that
   if $\End(E)=\Z[i]$ or $\End(E)=\Z[\frac12(1+i\sqrt3)]$, then
   the Seshadri constants on $E\times E$ are computed by elliptic
   curves, and hence they are integers.
   We now show that in all other cases there exist ample line
   bundles on $E\times E$,
   whose Seshadri constants cannot be computed by elliptic
   curves. With Theorem~\ref{thm:intro-computed-elliptic}
   this will imply that there are line bundles with fractional
   Seshadri constants on the surface.

\begin{proposition}\label{prop:exist-fractional}
   Let $E$ be an elliptic curve with complex multiplication.
   Write $\End(E)=\Z[\sigma]$ as above.
   If $\sigma\notin\set{i, \frac12(1+i\sqrt3)}$,
   then there exist ample line bundles $L$ on $E\times E$
   such that $\sqrt{L^2}$ is not an integer
   and such that $\eps(L)$ is not computed by an elliptic curve.
\end{proposition}

\begin{proof}
   Our strategy is to exhibit ample line bundles $L$ whose
   intersection number with any nef line bundle -- and therefore
   in particular with every elliptic curve -- is bigger than
   $\sqrt{L^2}$. For such $L$, the Seshadri constant
   cannot be computed by an elliptic curve,
   since $\eps(L)\le \sqrt{L^2}$
   (see \cite[Prop.~5.1.9]{Lazarsfeld:pos}).

   We first treat the case $\sigma=\sqrt{-e}$ with $e\neq 1$.
   For $k\in\Z$ consider the line bundle
   \be
      L_k:=\mathcal{O}_{E\times E} (2e\,F_1 + 2e\,F_2 + e\,\nabla + k\,\Sigma)\,.
   \ee
   As $L_k\cdot (F_1+F_2) = 4e$, the line bundle $L_k$ is ample
   if and only if $L_k^2=6e^2-2ek^2>0$.
   (This is a consequence of \cite[Cor.~4.3.3]{BL:CAV}.)
   Let $M$ be an arbitrary line bundle numerically written as
   $M\equiv a_1F_1+a_2F_2+a_3\nabla+a_4\Sigma$.
   Then the intersection number of $L_k$ and $M$ is given by
   \be
      L_k\cdot M = 2ea_2 + 2ea_1 - 2ea_3 - 2eka_4
      \,.
   \ee
   The crucial point in this construction is that
   $L_k\cdot M$ is a multiple of $2e$.
   So in particular the intersection number of $L_k$ with any elliptic curve
   on $E\times E$ is at least $2e$, if $L_k$ is ample.
   We will show that we can choose $k$ in such a way
   that
   \begin{itemize}\compact
   \item[(i)]
      $L_k$ is ample and $\sqrt{L_k^2}<2e$,
   \item[(ii)]
      $L_k^2$ is not a perfect square.
   \end{itemize}
   If this is achieved, then we have an ample line bundle as claimed in the proposition.

   Turning to the proof of that claim, note that (i)
   is equivalent to the condition that $k$ lies in the interval $(\sqrt{e},\sqrt{3e})$.
   So we have to show that if $e\ne 1$ then
   this interval contains an integer $k$ such that also condition (ii) is fulfilled.
   The subsequent Lemma \ref{lemma:square} shows that
   if $L_k^2$ is a perfect square, then $L_{k+1}^2$ cannot be a perfect square.
   As the interval $(\sqrt{e}, \sqrt{3e})$
   contains at least two integers when $e\ge 6$,
   we are thus reduced to treating
   the range $2\leq e \leq 5$.
   For these values of $e$ we can do explicit calculations, which
   show that integers $k$ as required exist:
   \begin{center}
     \begin{tabular}{c|c c c c}
       $e$& $2$ & $3$ & $4$ & $5$ \\
       \hline
       $k$ & $2$ & $2$ & $3$ & $3$ \\
       $L_k^2$ & $8$ & $30$ & $24$ & $60$ \\
      \end{tabular}
   \end{center}

   Now we treat the second case, i.e., $\sigma=\tfrac12(1+\sqrt{-e})$ with
   $ e\equiv 3 \tmod 4$ and $e\neq 3$.
   In this case, we consider for odd $k\in\Z$ the line bundles
   \be
      L_k:=\mathcal{O}_{E\times E}(2e\,F_1 + 2e\,F_2 + (e-k)\,\nabla + 2k\,\Sigma)\,.
   \ee
   Analogously to the case before, $L_k$ is ample if and only if
   $L_k^2=6e^2-2ek^2>0$. Since $k$ is odd,
   it follows that the intersection number of $L_k$ and $M$, which is given by
   \be
      L_k\cdot M = 2ea_2 + 2ea_1 - 2ea_3 - e(k+1)a_4
      \,,
   \ee
   is a multiple of $2e$.
   If $e\ne 3$, then
   it is possible to choose an integer
   $k\in (\sqrt{e},\sqrt{3e})$. This ensures that $L_k$ is ample and that $\sqrt{L_k^2}<2e$.
   (Note that the interval does not contain an odd integer for $e=3$.)

   By the subsequent Lemma \ref{lemma:square} we know that if $L_k^2$ is a perfect square
   then $L_{k+2}^2$ is not.
   Since the interval $(\sqrt{e}, \sqrt{3e})$ contains at least four integers when $e\ge 30$,
   we are reduced to the cases $7\leq e \leq 27$.
   We finish the proof by providing explicit values of $k$ for the remaining six cases:
   \begin{center}
      \begin{tabular}{c|c c c c c c}
         $e$ & $7$ & $11$ & $15$ & $19$ & $23$ & $27$ \\ \hline
         $k$ & $3$ & $5$ & $5$ & $5$ & $5$ & $7$ \\
         $L_k^2$ & $168$ & $176$ & $600$ & $1216$ & $2024$ & $1728$ \\
      \end{tabular}
   \end{center}
   This completes the proof of the proposition.
\end{proof}

\begin{lemma}\label{lemma:square}
   Let $e\ge 2$ be an integer and let $e=m^2\,n$ be its unique representation
   as a product of a square and a square-free integer. For positive integers $k$
   we define $A_k:=6e^2-2ek^2$. Then either $A_k$ or $A_{k+1}$
   is not a perfect square, and if furthermore $n\ge 3$, then either $A_k$ or $A_{k+2}$
   is not a perfect square.
\end{lemma}

\begin{proof}
   First, we treat the case $n\ge 3$. Suppose that $A_k=m^2n\,(6m^2n-2k^2)$ is a perfect square.
   Since $n$ is square-free, it follows that $6m^2n-2k^2 = nr^2$ for an integer
   $r$. We deduce that either $n$ or $n/2$ is a divisor of $k^2$, and hence
   it is a divisor of $k$. Consequently, neither $A_{k+1}$ nor $A_{k+2}$ can be a perfect square,
   because $n$ and $n/2$, respectively, cannot be a divisor of $k+1$ and $k+2$.

   Next, we consider the case $n=2$. Suppose that $A_k=4m^2\,(6m^2-k^2)$ is a perfect square. Then
   the factor $6m^2-k^2$ must itself be a perfect square, and
   this implies that
   the equation
   $k^2+r^2=6m^2$ has an integral solution $(k,r,m)$.
   By cancelling common factors we then find also
   a solution with $\gcd(k,r)=1$.
   We will now obtain a contradiction by
   considering the equation modulo 8:
   On the one hand,
   $r^2+k^2$ is either $1$, $2$ or $5$ modulo 8,
   and on the other hand $6m^2$ is either $0$ or $6$ modulo 8.

   Finally, we treat the case $n=1$. Suppose that $A_k=m^2\,(6m^2-2k^2)$ is a perfect square.
   As before, it follows that $6m^2-2k^2=r^2$ for some integer $r$.
   Assume by way of contradiction that $A_{k+1}=m^2\, (r^2-4k-2)$ is a perfect square as well.
   Then the factor $r^2-4k-2$ must be a perfect square as well.
   This, however, cannot happen because it is either $2$ or $3$ modulo 4.
\end{proof}

\begin{proof}[Proof of Theorem~\ref{thm:intro-integer}]
   The implication (ii) $\Rightarrow$ (i) follows from
   \cite{Bauer-Schulz}, where it is shown that
   all Seshadri
   constants are computed by elliptic curves in that case.
   Assume now that condition (i) holds.
   By Prop.~\ref{prop:exist-fractional}
   there exist ample line bundles $L$ whose Seshadri constant is not
   computed by elliptic curves and such that $\sqrt{L^2}$ is not
   an integer. Theorem~\ref{thm:intro-computed-elliptic}
   then shows that there are
   ample line bundles on $E\times E$ with fractional Seshadri constants.
\end{proof}

   The method of proof of Theorem~\ref{thm:intro-integer} shows the existence of line bundles
   with fractional Seshadri constants, but does not construct them explicitly.
   One idea to find such line bundles very concretely is
   to look for principal polarizations on $E\times E$.
   Those are either irreducible, i.e., they arise from a smooth curve of genus 2,
   or they correspond to a sum of two elliptic curves (see \cite[Thm.~2]{Weil:torelli}).
   In the irreducible case one has a fractional Seshadri constant
   $\eps(L)=\frac43$ by \cite[Prop.~2]{Steffens:remarks}.
   The problem of finding smooth genus two curves on $E\times E$ was first
   studied by Hayashida and Nishi
   in \cite{Hayashida-Nishi:genus-two}, where they show that
   if $\End(E)$ is isomorphic to the maximal order of $\Q(\sqrt{-m})$, then
   there exists such principal polarizations $L$ if and only if $m\notin\set{1,3,7,15}$.
   Note that this shows in particular that there are cases in which no such principal
   polarizations exist.
   We extend their result by
   exhibiting irreducible principal polarizations when
   $\End(E)=\Z[\sqrt{-e}]$ with $e\equiv 2,3$ (mod 4).
   (Note that these include non-maximal orders.)

\begin{proposition}\label{prop:exists-principal-polarization}
   Let $E$ be an elliptic curve with complex multiplication such that
   $\End(E)=\Z[\sqrt{-e}]$ with $e\equiv 2,3$ (mod 4).
   Then there exist irreducible principal polarizations $L$ on $E\times E$.
   In particular, we have $\eps(L)=\frac43$ for these line bundles.
\end{proposition}

\begin{proof}
   Note to begin with that for
   an irreducible principal polarization $L$
   one has $\eps(L)=\frac43$:
   This was first shown by Steffens~\cite{Steffens:remarks}
   when the Picard number is one, where the Seshadri constant
   is computed by a curve in $|4L|$.
   Thanks to the fact that this curve is irreducible,
   it also computes
   $\eps(L)$ in the general case by \cite[Lemma~6.2]{Bauer:sesh-alg-sf}.

   Turning to the proof of the proposition,
   we first treat the case $e\equiv 2$ (mod 4).
   Writing
   $e=4n+2$, consider the line bundle
   \be
      L_n:&=&\mathcal{O}_{E\times E}(2(n+1)F_1+2F_2+\nabla + \Sigma)\\
      &=&\mathcal{O}_{E\times E}(-(2n+1)F_1 + \Delta + \Gamma)
      \,.
   \ee
   It is a consequence of $L_n\cdot (F_1+F_2) = 2n+4>0$ and $L_n^2=2$ that
   $L_n$ is a principal polarization.
   Arguing as in the proof of Proposition~\ref{prop:exist-fractional}, it follows
   that the intersection number of $L_n$
   with any elliptic curve $N\subset E\times E$ is a multiple of $2$.
   So, $L\cdot N\ne 1$ and therefore $L$ must be irreducible.

   The case $e=4n+3$ can be dealt with analogously: In this case
   one can show that the line bundle
   \be
      L_n&:=&\mathcal{O}_{E\times E}(2(n+1)F_1+2F_2+ \Sigma)\\
      &=&\mathcal{O}_{E\times E}(-(2n+1)F_1+F_2 + \Gamma)
   \ee
   is an irreducible principal polarization.
\end{proof}

   Theorem~\ref{thm:intro-integer} provides
   the exact picture, on
   which surfaces $E\times E$
   fractional Seshadri constants
   occur.
   It is important to point out that
   on the other hand there are always line bundles whose
   Seshadri constant are integral -- in fact this happens for
   all bundles in the cone in $\NS(E\times E)$ generated by the classes of
   $F_1,F_2,\Delta,\Gamma$:

\begin{proposition}
   For line bundles
   \be
      L=\mathcal O_{E\times E}(a_1F_1+a_2F_2+a_3\Delta +a_4\Gamma)
   \ee
   with non-negative coefficients $a_i$, one has
   \be
      \eps(L) & = & \min\set{L\cdot F_1, L\cdot F_2, L\cdot \Delta, L\cdot \Gamma} \\
              & = & \min\{a_2+a_3+a_4, a_1+a_3+|\sigma|^2 a_4, \\
              & & \quad      a_1+a_2+|1-\sigma|^2 a_4, a_1+|\sigma|^2 a_2+|1-\sigma|^2 a_3\} \,.
   \ee
\end{proposition}

\begin{proof}
   Let $D$ be the divisor $a_1F_1+a_2F_2+a_3\Delta +a_4\Gamma$,
   and let $C$ be
   any irreducible curve $C$ passing through $0$,
   which is not a component of $D$.
   As $D$ is effective, we have
   \be
      \frac{L\cdot C}{\mult_0 C}=\frac{D\cdot C}{\mult_0 C} \ge \frac{\mult_0 D\cdot \mult_0 C}{\mult_0 C} &\ge& a_1+a_2+a_3+a_4 \\
      &\ge& a_2+a_3+a_4 = L\cdot F_1
      \,.
   \ee
   This implies that $\eps(L)$ is computed by one of
   the components of $D$.
   Their intersection numbers with $L$ are computed using Prop.~\ref{prop:intersection-matrix},
   and this yields the assertion of the proposition.
\end{proof}

   As the following example shows,
   the line bundles in the cone generated by the classes of
   $F_1,F_2,\Delta,\Gamma$
   are not the only ones with integral Seshadri constants.

\begin{example}
   Consider the line bundle $L=\mathcal O_{E\times E}(4F_1+2F_2-\Delta)$.
   It is ample by Prop.~\ref{prop:ample-cone}.
   The fact that $L\cdot F_1 =1$ implies
   that its Seshadri constant is $1$.
\end{example}

   Finally, we will discuss whether or not the statement in Theorem~\ref{thm:intro-computed-elliptic}
   can be generalized such that the conditions hold for each individual line bundle.
   Clearly, if there exists a line bundle $L$ with a fractional Seshadri constant,
   then a suitable multiple of $L$ will lead to a line bundle,
   whose Seshadri constant is an integer but is not calculated by an elliptic curve.
   One might hope that it still holds for primitive line bundles.
   This, however, is not the case:

\begin{proposition}\label{prop:primitive-lb}
   There exists an abelian surface $X$ and a primitive line bundle $L$ on $X$
   such that the Seshadri constant $\varepsilon(L)$ is an integer
   less than $\sqrt{L^2}$
   and is calculated by a non-elliptic curve.
\end{proposition}

\begin{proof}
   Let $E$ be an elliptic curve with complex multiplication
   such that $\End(E)=\Z[\sqrt{-2}]$.
   As noted in the proof of Prop.~\ref{prop:exists-principal-polarization},
   the Seshadri constant of the principal polarization
   $L_0=\mathcal{O}_{E\times E} (-F_1+\Delta+\Gamma)$ on $E\times E$
   is computed by an irreducible curve $C\in |4L_0|$ with $\mult_0(C)=6$.
   We consider the primitive line bundle
   $L:=\mathcal{O}_{E\times E}(D)$, where the divisor $D$ is defined as
   \be
      D&:=&3C+F_1+F_2+\Delta\\
      &=&-11F_1+F_2+13\,\Delta+12\,\Gamma\,.
   \ee
   We claim that the Seshadri constant $\varepsilon(L)$ equals $20$
   and is calculated by $C$.

   We have
   \be
      \frac{L\cdot D}{\mult_0(D)}=\frac{D^2}{\mult_0{D}}=\frac{438}{21}<\sqrt{438}=\sqrt{L^2}\,,
   \ee
   so $D$ is submaximal for $L$. Therefore \cite[Lemma~6.2]{Bauer:sesh-alg-sf} implies that the Seshadri constant of $L$ is calculated by a component of $D$.
   So checking $C$, $F_1$, $F_2$ and $\Delta$ we see that
   \be
      \varepsilon(L)=20=\frac{L\cdot C}{\mult_0 C}<26=L\cdot F_1=L\cdot F_2 = L\cdot \Delta\,.
   \ee
\end{proof}



\footnotesize
   \bigskip
   Thomas Bauer,
   Fachbereich Mathematik und Informatik,
   Philipps-Universit\"at Marburg,
   Hans-Meerwein-Stra\ss e,
   D-35032 Marburg, Germany.

   \nopagebreak
   \textit{E-mail address:} \texttt{tbauer@mathematik.uni-marburg.de}

   \bigskip
   Felix Fritz Grimm,
   Fachbereich Mathematik und Informatik,
   Philipps-Universit\"at Marburg,
   Hans-Meerwein-Stra\ss e,
   D-35032 Marburg, Germany.

   \nopagebreak
   \textit{E-mail address:} \texttt{felix@geller-grimm.de}

   \bigskip
   Maximilian Schmidt,
   Fachbereich Mathematik und Informatik,
   Philipps-Universit\"at Marburg,
   Hans-Meerwein-Stra\ss e,
   D-35032 Marburg, Germany.

   \nopagebreak
   \textit{E-mail address:} \texttt{schmid4d@mathematik.uni-marburg.de}


\end{document}